\documentclass[10pt]{article}
\usepackage{amsfonts}
\usepackage{amssymb}
\usepackage{amsthm}
\usepackage{amsmath}
\usepackage{bm}
\usepackage{amscd} 
\usepackage{MnSymbol}

\usepackage{tikz}

\makeatletter
\newbox\xrat@below
\newbox\xrat@above
\newcommand{\longrightarrowtail}[2][]{%
  \setbox\xrat@below=\hbox{\ensuremath{\scriptstyle #1}}%
  \setbox\xrat@above=\hbox{\ensuremath{\scriptstyle #2}}%
  \pgfmathsetlengthmacro{\xrat@len}{+2em}%
  \mathrel{\tikz [>->,baseline=-.75ex]
                 \draw (0,0) -- node[below=-2pt] {\box\xrat@below}
                                node[above=-2pt] {\box\xrat@above}
                       (\xrat@len,0) ;}}
\makeatother

\usepackage{hyperref}

\numberwithin{equation}{section} \DeclareMathSizes{2}{10}{12}{13}
\newtheorem{thm}{Proposition}[section]

\newtheorem{defn}[thm]{Definition}

\title{Dennis trace map for certain $K$-groups of categories with cofibrations }
\author{Abhishek Banerjee}

\date{}

\begin{document}

\maketitle

\centerline{\emph{Hausdorff Institute for Mathematics, Poppelsdorfer Allee 45,  D-53115, Bonn, Germany.}}
\centerline{\emph{Email: abhishekbanerjee1313@gmail.com}}
\medskip
\begin{abstract}

Let $\mathcal C$ be a small  category with cofibrations. In this paper, we define the $K$-theory and Hochschild
homology groups of $\mathcal C$ of order $Y$, where $Y$ is an ordered finite simplicial set with basepoint. Further, we construct
the Dennis trace map between these groups. 

\medskip
\centerline{\bf R\'{e}sum\'{e}}

\medskip
Soit $\mathcal C$ une petite cat\'{e}gorie avec cofibrations. Dans cet article, nous d\'{e}finissons ses groupes
de $K$-th\'{e}orie et homologie de Hochschild d'ordre $Y$, o\`{u} $Y$ est un ensemble simplicial ordonn\'{e} et fini avec point de base. 
De plus, nous construisons le morphisme de Dennis entre ces groupes. 

\end{abstract}

\medskip

\section{Introduction}

\medskip
In \cite{Wald}, Waldhausen introduced categories with cofibrations and defined their $K$-theory groups by means of the ``$S$-construction''. For any $n\geq 0$, Waldhausen's $S$-construction associates to a small category $\mathcal C$ with cofibrations a
category $S_n\mathcal C$. An object of $S_n\mathcal C$ is a chain of $n$ composable cofibrations  in $\mathcal C$ starting with the zero object (see \eqref{2.1}) Then, the objects of the simplicial category $S_\bullet\mathcal C=\{S_n\mathcal C\}_{n\geq 0}$ determine 
a simplicial set $S\mathcal C=\{obj(S_n\mathcal C)\}_{n\geq 0}$. Then, the  $K$-theory groups 
 of the category
$\mathcal C$ are defined to be the homotopy groups  of the loop space $\Omega|S\mathcal C|$ of the geometric realization of the 
simplicial set $S\mathcal C$. 

\medskip
Let $Ord_*$ denote the category of finite, totally ordered sets with basepoint. Let $SmCat_0$ denote the category of small categories with  zero objects. The starting point for this article is the fact that the association $n\mapsto S_n\mathcal C$ extends naturally to a functor $S(\mathcal C)$ from $Ord_*$ to $SmCat_0$. Then, if we take a simplicial object $Y:\Delta^{op}\longrightarrow Ord_*$ of
$Ord_*$, i.e., $Y$ is an ordered finite simplicial set with basepoint, we consider the composition (see \eqref{2.4})
\begin{equation}\label{1.2}
\begin{CD}
S^Y(\mathcal C):\Delta^{op}  @>Y>> Ord_* @>S(\mathcal C)>> SmCat_0 @>obj>> Sets_*
\end{CD}
\end{equation} The purpose of this paper is to study the  $K$-groups $K_p^Y(\mathcal C)$, $p\geq 0$ of $\mathcal C$ of order $Y$ which we define to be the homotopy groups
of the loop space of the geometric realization of the simplicial set $S^Y(\mathcal C)$ (see \eqref{2.5}). If $Y$, $Y':\Delta^{op}\longrightarrow Ord_*$ are simplicially homotopy equivalent as simplicial objects of $Ord_*$, we show that $K_p^{Y}(\mathcal C)\cong K_p^{Y'}(\mathcal C)$. Further, we  describe a product structure $K_p^Y(\mathcal C)\times K_q^Y(\mathcal D)\longrightarrow K_{p+q}^Y(\mathcal E)$ for a bi-exact functor $F:\mathcal C\times \mathcal D\longrightarrow \mathcal E$ (see \eqref{2.7qm}). 

\medskip

In the second part of the paper, we want to define Dennis trace maps from $K_p^Y(\mathcal C)$ to appropriate Hochschild homology
groups. For this, we consider the geometric realization $|CN(S^Y(\mathcal C))|$ of the bisimplicial set given by the cyclic nerve $CN(S^Y(\mathcal C))$ of the simplicial category $S(\mathcal C)\circ Y$. We define the Hochschild homology groups $HH_p^Y(\mathcal C)$ of $\mathcal C$ of order $Y$ in terms of the singular homology
of $|CN(S^Y(\mathcal C))|$ with coefficients in a given field $k$ (see \eqref{3.3df}). Again, if $Y$, $Y':\Delta^{op}\longrightarrow Ord_*$ are simplicially homotopy equivalent as simplicial objects of $Ord_*$, we show that $HH_p^{Y}(\mathcal C)\cong HH_p^{Y'}(\mathcal C)$. We show that the Hochschild homology groups also carry a product $HH_p^Y(\mathcal C)\otimes HH_q^Y(\mathcal D)
\longrightarrow HH_{p+q+1}^Y(\mathcal E)$ for a bi-exact functor $F:\mathcal C\times \mathcal D\longrightarrow \mathcal E$ (see \eqref{3.6cv}). Finally, for any $p\geq 0$, we construct a Dennis trace map
$D_p^Y:K_p^Y(\mathcal C)\longrightarrow HH_p^Y(\mathcal C)$.

\section{The $K$-groups $K_p^Y(\mathcal C)$ for a category with cofibrations}

\medskip
 In this section and throughout this paper, we let $\mathcal C$ be a small category with cofibrations in the sense of 
Waldhausen \cite{Wald}. In other words, $\mathcal C$ is a category with a zero object together with a subcategory $co\mathcal C$ 
satisfying the axioms (Cof1) and (Cof2) below. The morphisms in $co\mathcal C$ will be referred to as cofibrations
and denoted by feathered arrows ``$\rightarrowtail$''. 

\medskip
(Cof1) Every isomorphism in $\mathcal C$ is a cofibration. For any object $A$ in $\mathcal C$, the canonical morphism
$0\longrightarrow A$ is a cofibration.

(Cof2) Given a cofibration $A\rightarrowtail B$, its pushout $C\coprod_AB$ along any other morphism 
$A\longrightarrow C$ exists in $\mathcal C$ and the canonical morphism $C\longrightarrow C\coprod_AB$
is a cofibration. 

\medskip
Given a cofibration $A\rightarrowtail B$ in $\mathcal C$, its pushout 
$0\coprod_AB$ along the morphism $A\longrightarrow 0$ will be denoted by $B/A$. The canonical morphism
from $B$ to the pushout $B/A=0\coprod_AB$ is referred to as a quotient map and denoted by $B\twoheadrightarrow B/A$. 
The sequence $A\rightarrowtail B\twoheadrightarrow B/A$ is referred to as a cofibration sequence. A functor
between categories with cofibrations is said to be exact if it takes $0$ to $0$, preserves cofibrations 
as well as the pushout diagrams arising from axiom (Cof2). 

\medskip
Given $\mathcal C$ as above, we let $S_\bullet\mathcal C$ be the simplicial category associated to $\mathcal C$ by
Waldhausen's $S$-construction (see \cite[1.3]{Wald}). More explicitly, for any $n\geq 0$, an object of $S_n\mathcal C$ is a sequence
$(a_0,a_1,...,a_{n-1})$ of composable cofibrations:
\begin{equation}\label{2.1}
0=A_{0}\longrightarrowtail[]{a_0} A_{1}\longrightarrowtail[]{a_1}A_{2}\longrightarrowtail[]{a_2}\dots 
\longrightarrowtail[]{a_{n-1}} A_{n}
\end{equation} together with a choice of quotients $A_{ij}=A_{j}/A_{i}$. Further, for any $i\leq j\leq k$, 
$A_{ij}\rightarrowtail A_{ik}\twoheadrightarrow A_{jk}$ is a cofibration sequence. 

\medskip
We now let $Ord_*$ denote the category of finite totally ordered sets $(Z,\ast)=\{\ast < z_1<z_2<....<z_k\}$ with basepoint
$\ast$. A morphism $\phi:(Z,\ast)\longrightarrow (Z',\ast)$ in $Ord_*$ satisfies $\phi(\ast)=\ast$ and $\phi(x)\leq \phi(y)\in Z'$ for
any $x\leq y$ in $Z$. Let $\overset{\rightarrow}{\Gamma}_*$ denote the subcategory of $Ord_*$  consisting of the objects $[n]=\{0<1<2<...<n\}$ (with basepoint $0$) for 
any $n\geq 0$. Then, the category $\mathcal C$ determines a functor:
\begin{equation}\label{2.2}
S(\mathcal C):\overset{\rightarrow}{\Gamma}_*\longrightarrow SmCat_0 \qquad [n]\mapsto S_n\mathcal C
\end{equation}  where $SmCat_0$ denotes the category of   small categories with  zero objects. The morphisms in $SmCat_0$  are functors that preserve zero objects.  Given a morphism $\phi:[n]
\longrightarrow [m]$ in $\overset{\rightarrow}{\Gamma}_*$, we have an induced functor:
\begin{equation}\label{2.3}
S(\phi):S_n\mathcal C\longrightarrow S_m\mathcal C \qquad (a_0,...a_{n-1})\mapsto (b_0,...,b_{m-1}) \qquad b_j:=\prod_{i\in \phi^{-1}(j),i<n}a_i
\end{equation} where in \eqref{2.3}, for any $0\leq j\leq m-1$, the cofibration $b_j=\prod_{i\in \phi^{-1}(j),i<n}a_i$ is the composition of the cofibrations $a_i$ where $i$ lies in the ordered set $\phi^{-1}(j)$ and $i<n$. In \eqref{2.3} it is understood that when $\phi^{-1}(j)\cap \{0,1,2,...,n-1\}$ is empty, we set
$b_j=1$. It is clear that the functor $S(\mathcal C):\overset{\rightarrow}{\Gamma}_*\longrightarrow SmCat_0$ in \eqref{2.2}
extends to a functor from $Ord_*$ to $SmCat_0$ that we continue to denote by $S(\mathcal C):Ord_*\longrightarrow SmCat_0$. We are
now ready to define the $K$-groups of $\mathcal C$ with respect to an ordered finite simplicial set $Y$ with basepoint. 

\medskip
\begin{defn} Let $\mathcal C$ be a category as above and let $Y:\Delta^{op}\longrightarrow Ord_*$ be an ordered finite simplicial
set with basepoint. Let $Sets_*$ denote the category of pointed sets. We consider the following composition of functors:
\begin{equation}\label{2.4}
\begin{CD}
S^Y(\mathcal C):\Delta^{op}  @>Y>> Ord_* @>S(\mathcal C)>> SmCat_0 @>obj>> Sets_*
\end{CD}
\end{equation} where $obj:SmCat_0\longrightarrow Sets_*$ is the functor that associates a category in $SmCat_0$ to its
set of objects (with the zero object going to the basepoint). We consider the geometric realization $|S^Y(\mathcal C)|$ of the pointed simplicial set $S^Y(\mathcal C)$ in \eqref{2.4}
and its loop space $\Omega|S^Y(\mathcal C)|$. Then, we define  $K$-theory groups $K^Y_p(\mathcal C)$ of the category $\mathcal C$ of order
$Y$ to be the homotopy groups:
\begin{equation}\label{2.5}
K_p^Y(\mathcal C):=\pi_p(\Omega|S^Y(\mathcal C)|) \qquad \forall \textrm{ }p\geq 0
\end{equation}

\end{defn}

\medskip
We now show that homotopic maps of ordered simplicial sets determine identical morphisms on the $K$-groups
defined above.

\medskip
\begin{thm}\label{P2.2} Let $\mathcal C$ be a small category with cofibrations. Let $Y$, $Y':\Delta^{op}\longrightarrow Ord_*$ be ordered finite simplicial sets
with basepoint. Let $f,g:Y\longrightarrow Y'$ be morphisms of simplicial objects of 
$Ord_*$ that are simplicially homotopic. Then, $f$ and $g$ induce identical morphisms
$K(f)_p=K(g)_p:K_p^Y(\mathcal C)\longrightarrow K_p^{Y'}(\mathcal C)$ $\forall$ $p\geq 0$. In particular, if $Y$ and $Y'$ are 
simplicially homotopy equivalent as simplicial objects of $Ord_*$, $K_p^Y(\mathcal C)\cong K^{Y'}_p(\mathcal C)$. 
\end{thm}

\begin{proof}   Let $f_n,g_n:Y_n\longrightarrow Y'_n$ be the morphisms corresponding to $f$ and $g$ respectively at each level $n$. We are given that $f$ and $g$ are simplicially homotopic morphisms between simplicial objects 
of $Ord_*$. It follows that (see \cite[$\S$ 8.3.11]{Weibel}) there are morphisms $h_{i,n}:Y_n\longrightarrow Y'_{n+1}$, 
$0\leq i\leq n$, $n\geq 0$ in $Ord_*$ such that $d_{0,n+1}^{Y'}h_{0,n}=f_n$ and $d^{Y'}_{n+1,n+1}h_{n,n}=g_n$ and 
\begin{equation}\label{2.6}
\begin{array}{ll}
d_{i,n+1}^{Y'}h_{j,n} &= \left\{ 
\begin{array}{ll}
h_{j-1,n-1}d^Y_{i,n} & \mbox{if $i<j$} \\
d^{Y'}_{i,n+1}h_{i-1,n} & \mbox{if $i=j\ne 0$} \\
h_{j,n-1}d^Y_{i-1,n} & \mbox{if $i\geq j+1$} \\
\end{array} \right. \\
s^{Y'}_{i,n+1}h_{j,n} & = \left\{ 
\begin{array}{ll}
h_{j+1,n+1}s_{i,n}^Y & \mbox{if $i\leq j$} \\
h_{j,n+1}s_{i-1,n}^Y & \mbox{if $i>j$} \\
\end{array}\right.  \\
\end{array}
\end{equation} Here $d_{i,n}^Y:Y_n\longrightarrow Y_{n-1}$ (resp. $d_{i,n}^{Y'}:Y'_n\longrightarrow Y'_{n-1}$)
and $s_{i,n}^Y:Y_n\longrightarrow Y_{n+1}$ (resp. $s^{Y'}_{i,n}:Y'_n\longrightarrow Y'_{n+1}$) for $0\leq i\leq n$ are respectively
the face and degeneracy maps of the simplicial object $Y$ (resp. $Y'$) of $Ord_*$. We now consider the simplicial
sets $S^{Y}(\mathcal C)=\{S^Y(\mathcal C)_n\}_{n\geq 0}$ and $S^{Y'}(\mathcal C)=\{S^{Y'}(\mathcal C)_n\}_{n\geq 0}$ as defined in \eqref{2.4}. By definition, for any $n\geq 0$,
$S^Y(\mathcal C)_n=obj(S(\mathcal C)(Y_n))$ and $S^{Y'}(\mathcal C)_n=obj(S(\mathcal C)(Y'_n))$
along with induced maps $obj(S(\mathcal C)(f_n)), obj(S(\mathcal C)(g_n)):S^Y(\mathcal C)_n\longrightarrow S^{Y'}(\mathcal C)_n$. Then, the induced
maps $obj(S(\mathcal C)(h_{i,n})):S^Y(\mathcal C)_n\longrightarrow S^{Y'}(\mathcal C)_{n+1}$ 
define a simplicial homotopy between the two maps $S^f(\mathcal C)=\{obj(S(\mathcal C)(f_n))\}_{n\geq 0}$ and $
S^g(\mathcal C)=\{obj(S(\mathcal C)(g_n))\}_{n\geq 0}:S^Y(\mathcal C)\longrightarrow 
S^{Y'}(\mathcal C)$ of simplicial sets. 

\medskip
It now follows from the definitions in \eqref{2.5} that the induced morphisms $K(f)_p:=\pi_p(\Omega |S^f(\mathcal C)|), K(g)_p:=\pi_p(\Omega |S^g(\mathcal C)|):K_p^Y(\mathcal C)=\pi_p(\Omega |S^Y(\mathcal C)|)
\longrightarrow \pi_p(\Omega |S^{Y'}(\mathcal C)|)=K_p^{Y'}(\mathcal C)$  on the homotopy groups are identical.

\end{proof}

Given small categories with cofibrations $\mathcal C$, $\mathcal D$, we consider the category $\mathcal C\times 
\mathcal D$. An object of $\mathcal C\times \mathcal D$ is a pair $(C,D)$ where $C\in obj(\mathcal C)$ and $
D\in obj(\mathcal D)$. For $(C,D)$, $(C',D')\in obj(\mathcal C\times \mathcal D)$, the collection of morphisms from
$(C,D)$ to $(C',D')$ in $\mathcal C\times \mathcal D$ is given by $Hom_{\mathcal C}(C,C')\times Hom_{\mathcal D}(D,D')$. Given a small category with cofibrations $\mathcal E$, we now recall
that a functor $F:\mathcal C\times \mathcal D\longrightarrow \mathcal E$ is said to be bi-exact if it satisfies the following
two conditions (see, for instance, \cite[Definition 4.2.1]{RM}):

(1) For any $C\in obj(\mathcal C)$ (resp. $D\in obj(\mathcal D)$), the functor $F(C,\_\_):\mathcal D\longrightarrow 
\mathcal E$ (resp. 
$F(\_\_,D):\mathcal C\longrightarrow \mathcal E$) is exact.

(2) Given cofibrations $C\rightarrowtail C'$ and $D\rightarrowtail D'$ in the categories $\mathcal C$ and $
\mathcal D$ respectively, the canonical morphism from $F(C',D)\coprod_{F(C,D)}F(C,D')$ to  
$F(C',D')$ is a cofibration in $\mathcal E$. 

\begin{thm}\label{P2.3}
Let $Y:\Delta^{op}\longrightarrow Ord_*$ be an ordered finite simplicial
set with basepoint. Let $\mathcal C$, $\mathcal D$ and $\mathcal E$ be  small categories with cofibrations and let $F:
\mathcal C\times \mathcal D\longrightarrow \mathcal E$ be a bi-exact functor. 
Then, there exists a product structure:
\begin{equation}\label{2.7qm}
K_p^{Y}(\mathcal C)\times K_{q}^Y(\mathcal D)\longrightarrow K_{p+q}^Y(\mathcal E)\qquad\forall
\textrm{ }p,q\geq 0
\end{equation} 
\end{thm}

\begin{proof} Given a bi-exact functor $F:\mathcal C\times \mathcal D\longrightarrow \mathcal E$, we consider the induced 
functors $F_{n}:S_n\mathcal C\times S_n\mathcal D\longrightarrow S_n\mathcal E$, $n\geq 0$ defined as follows:
\begin{equation} \label{2.9vp}
\begin{CD}
(0=A_{0}\rightarrowtail A_{1}\rightarrowtail \dots \rightarrowtail A_{n})\times (0=B_{0}\rightarrowtail B_{1}\rightarrowtail \dots \rightarrowtail
B_{n})\in S_n\mathcal C\times S_n\mathcal D \\ @VVV \\
(0=F(A_{0},B_{0})  \longrightarrowtail[]{}  F(A_{1},B_{1})  \longrightarrowtail[]{}  \dots  \longrightarrowtail[]{} F(A_{n},B_{n}))\in S_n\mathcal E 
\\
\end{CD}
\end{equation} We note that since $F$ is bi-exact, the  morphisms $F(A_{k},B_{k})\longrightarrow F(A_{k+1},B_{k})$ and 
$F(A_{k+1},B_k)\longrightarrow F(A_{k+1},B_{k+1})$ are cofibrations in $\mathcal E$ for each $k\geq 0$. Hence, 
each morphism $F(A_k,B_k)\longrightarrow F(A_{k+1},B_{k+1})$ in \eqref{2.9vp} obtained by composing
$F(A_{k},B_{k})\longrightarrow F(A_{k+1},B_{k})$ and 
$F(A_{k+1},B_k)\longrightarrow F(A_{k+1},B_{k+1})$ is a cofibration. 
 Using  the fact that $F$ is bi-exact, we also see  that $F(A_k,0)=F(0,B_k)=0$. Hence, it
follows that \eqref{2.9vp} induces a morphism $obj(S_n\mathcal C)\wedge
obj(S_n\mathcal D)\longrightarrow obj(S_n\mathcal E)$ of pointed sets. Then, if we consider the ordered finite simplicial set $Y=\{Y_n\}_{n\geq 0}$,   we have morphisms:
\begin{equation}\label{2.10}
F^Y_{n}:obj(S(\mathcal C)(Y_n))\wedge obj(S(\mathcal D)(Y_n))
\longrightarrow obj(S(\mathcal E)(Y_n)) \qquad\forall\textrm{ }n\geq 0
\end{equation} From \eqref{2.10}, it follows that we have a morphism $F^Y:S^Y(\mathcal C)\wedge 
S^Y(\mathcal D)\longrightarrow S^Y(\mathcal E)$ of pointed simplicial sets. Passing to geometric realizations and taking loop spaces, we have an induced map
$\Omega |S^Y(\mathcal C)|\wedge  \Omega |S^Y(\mathcal D)|\longrightarrow \Omega |S^Y(\mathcal E)|$. The result is now clear from the definitions
in \eqref{2.5}.

\end{proof}

\medskip

\medskip
\section{Hochschild homology and the Dennis trace map}

\medskip

\medskip
We recall that a cyclic set is a contravariant functor from Connes' cyclic category $\Delta C$ to the category $Sets$ of
sets (for details see, for instance, \cite[$\S$ 6.1.2.1]{Lod}). 
Given a small category $\mathcal A$, we can associate to it the cyclic set $CN(\mathcal A)=\{CN_n(\mathcal A)\}_{n\geq 0}$ given by its cyclic nerve; in other
words, for any $n\geq 0$, we set:
\begin{equation}
CN_n(\mathcal A):=\underset{(A_0,...,A_n)\in obj(\mathcal A)^{n+1}}{\coprod} Hom_{\mathcal A}(A_1,A_0)\times Hom_{\mathcal A}(A_2,A_1)\times \dots \times 
Hom_{\mathcal A}(A_0,A_n)
\end{equation} By abuse of notation, given a small category $\mathcal A$, we will also let $CN(\mathcal A)=\{CN_n
(\mathcal A)\}_{n\geq 0}$ denote
the underlying simplicial set of the cyclic set $CN(\mathcal A)$. 

\medskip
\begin{defn}\label{D3.1}  Let $\mathcal C$ be a small category with cofibrations  and let $Y:\Delta^{op}\longrightarrow Ord_*$ be an ordered
finite simplicial set with basepoint. Let $k$ be a given field. We consider the composition of functors:
\begin{equation}\label{3.2}
\begin{CD}
CN(S^Y(\mathcal C)):\Delta^{op} @>Y>> Ord_* @>S(\mathcal C)>> SmCat_0 @>CN>> SSets \\ 
\end{CD}
\end{equation} where $SSets$ denotes the category of simplicial sets.  Let $|CN(S^Y(\mathcal C))|$ denote the geometric realization of the bisimplicial set 
$CN(S^Y(\mathcal C))$. Then, we define the Hochschild homologies $HH_p^Y(\mathcal C)$ of the category $\mathcal C$
of order $Y$ over the field $k$ to be the homology groups:
\begin{equation}\label{3.3df}
HH_p^Y(\mathcal C):=H_{p+1}(|CN(S^Y(\mathcal C))|,k)\qquad \forall\textrm{ }p\geq 0
\end{equation} 

\end{defn}

\medskip
As noted before, the cyclic nerve of a small category is a cyclic set. As such,  the bisimplicial set 
$CN(S^Y(\mathcal C))$ in \eqref{3.2} is actually a ``cyclic $\times$ simplicial set'' (i.e., a cyclic set
in one coordinate and a simplicial set in the other; see, for instance, \cite[Appendix A.6]{RM}). Taking the geometric realization first in the simplicial direction, we obtain 
a cyclic space whose geometric realization carries the structure of an $S^1$-space (see \cite[$\S$ 7.1.4]{Lod}). Then,  we can consider the cyclic  geometric
realization $|CN(S^Y(\mathcal C))|^{cy}$ of $CN(S^Y(\mathcal C))$  which is given by the Borel space $|CN(S^Y(\mathcal C))|^{cy} :
=ES^1\times_{S^1}|CN(S^Y(\mathcal C))|$ (see \cite[$\S$ 7.2.2]{Lod}). Here $ES^1$ is any contractible
space on which the  topological group $S^1$ has a free action. We can define the cyclic homologies
$HC_p^Y(\mathcal C)$ of
the category $\mathcal C$ of order $Y$ to be the homology groups  $HC_p^Y(\mathcal C):=H_{p+1}(|CN(S^Y(\mathcal C))|^{cy},k)$, $ \forall\textrm{ }p\geq 0$. 

\medskip
\begin{thm}  Let $\mathcal C$ be a small category with cofibrations  and let $Y:\Delta^{op}\longrightarrow Ord_*$ be an ordered
finite simplicial set with basepoint. Then, the Hochschild and cyclic homologies of $\mathcal C$ of order $Y$ fit into a long exact sequence: 
\begin{equation}\label{3.3}
\dots \longrightarrow HH_p^Y(\mathcal C)\longrightarrow HC^Y_p(\mathcal C) \longrightarrow HC^Y_{p-2}(\mathcal C)
\longrightarrow HH_{p-1}^Y(\mathcal C)\longrightarrow \dots 
\end{equation}

\end{thm}

\begin{proof} From \cite[$\S$ 7.2.7]{Lod}, it follows that there exists a homotopy fibration
$|CN(S^Y(\mathcal C))|\longrightarrow |CN(S^Y(\mathcal C))|^{cy}\longrightarrow BS^1$, where $BS^1$ is the classifying
space of the topological group $S^1$. Hence, it follows from \cite[$\S$ 7.2.10]{Lod} that the long exact  sequence 
given by the homology spectral sequence corresponding to this fibration gives us the long exact sequence in
\eqref{3.3}. 

\end{proof} 

\begin{thm}\label{P3.3} Let $\mathcal C$ be a small category with cofibrations. Let $Y$, $Y':\Delta^{op}\longrightarrow Ord_*$ be ordered finite simplicial sets
with basepoint. Let $f,g:Y\longrightarrow Y'$ be morphisms of simplicial objects of 
$Ord_*$ that are simplicially homotopic. Then, $f$ and $g$ induce identical morphisms
$H(f)_p=H(g)_p:HH_p^Y(\mathcal C)\longrightarrow HH_p^{Y'}(\mathcal C)$, $\forall$ $p\geq 0$. In particular, if $Y$, $Y'$ are simplicially 
homotopy equivalent as simplicial objects of $Ord_*$, $HH_p^Y(\mathcal C)\cong HH^{Y'}_p(\mathcal C)$. 
\end{thm}

\begin{proof}For any fixed $n\geq 0$, we consider the simplicial set $CN_n(S^Y(\mathcal C))$ given by the composition: \begin{equation}\label{3.5}
\begin{CD}
CN_n(S^Y(\mathcal C)):\Delta^{op} @>Y>> Ord_* @>S(\mathcal C)>> SmCat_0 @>CN_n>> Sets \\ 
\end{CD}
\end{equation} Then, as in the proof of Proposition \ref{P2.2}, it follows that the maps of simplicial sets 
$CN_n(S^f(\mathcal C))$, $CN_n(S^g(\mathcal C)):CN_n(S^Y(\mathcal C))\longrightarrow CN_n(S^{Y'}(\mathcal C))$ induced by
$f$ and $g$ respectively
are simplicially homotopic. 

\medskip
If we consider the geometric realization of the cyclic $\times$ simplical set $CN(S^Y(\mathcal C))$ (resp. 
$CN(S^{Y'}(\mathcal C))$ in the simplicial direction, we obtain the cyclic space $\{|CN_n(S^Y(\mathcal C))|\}_{n\geq 0}$
(resp. $\{|CN_n(S^{Y'}(\mathcal C))|\}_{n\geq 0}$). Here, the space  $|CN_n(S^Y(\mathcal C))|$ (resp. 
$|CN_n(S^{Y'}(\mathcal C))|$) is given by the geometric realization of the simplicial set $CN_n(S^Y(\mathcal C))=\{CN_n(S(\mathcal C)(Y_m))\}_{m\geq 0}$ (resp. 
$CN_n(S^{Y'}(\mathcal C))=\{CN_n(S(\mathcal C)(Y'_m)\}_{m\geq 0}$). From the above, it follows that the morphisms between the cyclic spaces 
$\{|CN_n(S^Y(\mathcal C))|\}_{n\geq 0}$
and $\{|CN_n(S^{Y'}(\mathcal C))|\}_{n\geq 0}$ induced respectively by $f$ and $g$ are homotopic in each degree. Hence, the morphisms between the geometric realizations of the cyclic $\times$ simplicial sets $CN(S^Y(\mathcal C))$ and 
$CN(S^{Y'}(\mathcal C))$ induced by $f$ and $g$ respectively are homotopic. It follows from \eqref{3.3df} that 
the induced maps $H(f)_p,H(g)_p:HH_p^Y(\mathcal C)\longrightarrow HH_p^{Y'}(\mathcal C)$ on the Hochschild homologies 
are identical. 

\end{proof} 

From the proof of Proposition \ref{P3.3} and the definition $HC_p^Y(\mathcal C):=H_{p+1}(|CN(S^Y(\mathcal C))|^{cy},k)$, it is clear that given simplicially homotopic 
maps $f,g:Y\longrightarrow Y'$ as above, an analogous result holds
for induced maps on cyclic homologies. 

\medskip
\begin{thm}Let $Y:\Delta^{op}\longrightarrow Ord_*$ be an ordered finite simplicial
set with basepoint. Let $\mathcal C$, $\mathcal D$ and $\mathcal E$ be  small categories with cofibrations and let $F:
\mathcal C\times \mathcal D\longrightarrow \mathcal E$ be a bi-exact functor. 
Then, there exists a product:
\begin{equation}\label{3.6cv}
HH_p^{Y}(\mathcal C)\otimes HH_{q}^Y(\mathcal D)\longrightarrow HH_{p+q+1}^Y(\mathcal E)\qquad\forall
\textrm{ }p,q\geq 0
\end{equation} 

\end{thm}

\begin{proof} As in the proof of Proposition \ref{P2.3}, for any $n\geq 0$, we have a functor $F_n:S_n\mathcal C
\times S_n\mathcal D\longrightarrow S_n\mathcal E$ induced by $F:\mathcal C\times \mathcal D\longrightarrow \mathcal E$. 
Given $Y=\{Y_n\}_{n\geq 0}:\Delta^{op}\longrightarrow Ord_*$, it is clear that the functors $F_n$, $n\geq  0$ induce  
$F^Y_n:S(\mathcal C)(Y_n)\times S(\mathcal D)(Y_n)\longrightarrow S(\mathcal E)(Y_n)$. Then,
for any $m\geq 0$, we have a map:
\begin{equation}\label{3.7}
\begin{array}{c}
CN_m(F^Y_n):CN_m(S(\mathcal C)(Y_n))\times CN_m(S(\mathcal D)(Y_n)) \longrightarrow CN_m(S(\mathcal E)(Y_n)) \\
\begin{CD}
(C_0\rightarrow C_n\rightarrow C_{n-1}\rightarrow ...\rightarrow C_1\rightarrow C_0)\times 
(D_0\rightarrow D_n\rightarrow D_{n-1}\rightarrow ...\rightarrow D_1\rightarrow D_0)
\\ @VVV  \\
 (F(C_0, D_0)\rightarrow F(C_n, D_n)\rightarrow F(C_{n-1}, D_{n-1})\rightarrow ...\rightarrow 
F(C_1, D_1)\rightarrow F(C_0, D_0))\\
\end{CD}
\end{array}
\end{equation} The map of bisimplicial sets in \eqref{3.7} induces a morphism $|CN(F^Y)|:|CN(S^Y(\mathcal C))|
\times |CN(S^Y(\mathcal D))|\longrightarrow |CN(S^Y(\mathcal E))|$ of geometric realizations. Finally, this gives us a 
product on homologies ($\forall$ $p$, $q\geq 0$):
\begin{equation}\label{3.8bpu}
H_{p+1}(|CN(S^Y(\mathcal C))|,k)
\otimes H_{q+1}(|CN(S^Y(\mathcal D))|,k)\longrightarrow H_{p+q+2}(|CN(S^Y(\mathcal E))|,k)
\end{equation} Comparing \eqref{3.8bpu} with the definitions in \eqref{3.3df}, we obtain the map in \eqref{3.6cv}.

\end{proof}

\begin{thm} Let $\mathcal C$ be a small category with cofibrations and let $Y:\Delta^{op}\longrightarrow Ord_*$ be an ordered
finite simplicial set with basepoint.  Then, for each $p\geq 0$, there is a morphism $D_p^Y:K_p^Y(\mathcal C)
\longrightarrow HH_p^Y(\mathcal C)$ from the $K$-groups of $\mathcal C$ of order $Y$ to its Hochschild homology
groups. 
\end{thm}

\begin{proof} Given a small category $\mathcal A$ in $SmCat_0$, for any $n\geq 0$, we have a map 
$CN_0(\mathcal A)\longrightarrow CN_n(\mathcal A)$ of sets that takes any $A\in obj(\mathcal A)=CN_0(
\mathcal A)$ to $(A\overset{1}{\longrightarrow} A \overset{1}{\longrightarrow} A \overset{1}{\longrightarrow}...
\overset{1}{\longrightarrow}A)\in CN_n(\mathcal A)$ (map $A\overset{1}{\longrightarrow}A$ repeated $n$ times). This gives us a morphism $CN_0(\mathcal A)\longrightarrow CN(
\mathcal A)$ of simplicial sets (where $CN_0(\mathcal A)$ is treated as a constant simplicial set) and hence a morphism
$CN_0\longrightarrow CN$ of functors from $SmCat_0$ to $SSets$. Composing with $S(\mathcal C)\circ Y:\Delta^{op}
\longrightarrow SmCat_0$, we have a morphism $D^Y:S^Y(\mathcal C)\longrightarrow CN(S^Y(\mathcal C))$ of functors from 
$\Delta^{op}\longrightarrow SSets$. The latter induces a map $|D^Y|:|S^Y(\mathcal C)|\longrightarrow |CN(S^Y(\mathcal C))|$ 
of geometric realizations. Combining with the definitions in  \eqref{2.5} and \eqref{3.3df},  we have a morphism
$D_p^Y:K_p^Y(\mathcal C)\longrightarrow HH_p^Y(\mathcal C)$ given by the composition:
\begin{equation*}
\begin{CD}
K_p^Y(\mathcal C)=\pi_p(\Omega |S^Y(\mathcal C)|)=\pi_{p+1}(|S^Y(\mathcal C)|) @>\pi_{p+1}(|D^Y|)>>\pi_{p+1}(|CN(S^Y(\mathcal C))|)
\rightarrow H_{p+1}(|CN(S^Y(\mathcal C))|,k)=HH_p^Y(\mathcal C) \\
\end{CD}
\end{equation*}
\end{proof}

\end{document}